\documentclass{amsart}

\newtheorem{theorem}{Theorem}[section]
\newtheorem{lemma}[theorem]{Lemma}
\newtheorem{cor}[theorem]{Corollary}

\theoremstyle{definition}
\newtheorem{definition}[theorem]{Definition}

\theoremstyle{remark}
\newtheorem{remark}[theorem]{Remark}

\numberwithin{equation}{section}

\begin{document}

\title[Hemi-slant submanifolds]{The geometry of hemi-slant submanifolds\\
of a locally product Riemannian manifold}

\author{Hakan Mete TA\c STAN}

\address{\.Istanbul University\\
Department of Mathematics\\
Vezneciler, \.Istanbul, Turkey}

\email{hakmete@istanbul.edu.tr}

\author{FATMA \"{O}ZDEM\.{I}R}
\address{Department of Mathematics\\
\.{I}stanbul Technical University\\Maslak,
\.{I}stanbul, Turkey}
\email{fozdemir@itu.edu.tr}


\begin{abstract}
In the present paper, we study hemi-slant submanifolds of a locally
product Riemannian manifold. We prove that the anti-invariant
distribution which is  involved in the definition of hemi-slant
submanifold is integrable and give some applications of this result.
We get a necessary and sufficient condition for a proper hemi-slant
submanifold to be a hemi-slant product. We also study this type
submanifolds with parallel canonical structures. Moreover, we give
two characterization theorems for the totally umbilical proper
hemi-slant submanifolds. Finally, we obtain a basic inequality
involving Ricci curvature and the squared mean curvature of a
hemi-slant submanifold of a certain type locally product Riemannian
manifold.
\end{abstract}
\maketitle

\section{Introduction}
Study of slant submanifolds was initiated by B.Y. Chen \cite{Che},
as a generalization of both holomorphic and totally real
submanifolds of a K\"{a}hler manifold. Slant submanifolds have been
studied in different kind structures; almost contact \cite{Lo},
neutral K\"{a}hler \cite{Ar}, Lorentzian Sasakian \cite{Ale} and
Sasakian \cite{C} by several geometers. N. Papaghiuc \cite{Papa}
introduced semi-slant submanifolds of a K\"{a}hler manifold as a
natural generalization of slant submanifold. A. Carriazo \cite{Ca},
introduced bi-slant submanifolds of an almost Hermitian manifold as
a generalization of semi-slant submanifolds. One of the classes of
bi-slant submanifolds is that of anti-slant submanifolds which are
studied by A. Carriazo \cite{Ca}. However, B. \c{S}ahin \cite{Sa}
called these submanifolds as hemi-slant submanifolds because of that
the name anti-slant seems to refer that it has no slant factor. We
observe that a hemi-slant submanifold is a special case of generic
submanifold which was introduced by G.S. Ronsse \cite{Ro}. Since
then many geometers have studied hemi-slant submanifolds in
different kind structures; K\"{a}hler \cite{Al,Sa}, nearly
K\"{a}hler \cite{Ud}, generalized complex space form \cite{Muku} and
almost Hermitian \cite{Ta}. We note that sometimes  hemi-slant submanifolds
are also studied under the name pseudo-slant submanifolds, see \cite{Kh} and \cite{Ud}.
The submanifolds of a locally product Riemannian manifold have been studied by many geometers.
For example, T. Adati \cite{A} defined and studied invariant and
anti-invariant submanifolds, while A. Bejancu \cite{Be} and  G.
Pitis \cite{Pi} studied semi-invariant submanifolds. Slant and
semi-slant submanifolds of a locally product Riemannian manifold are
examined by B. \c{S}ahin \cite{S} and H. Li and X. Liu \cite{Li}. In
this paper, we study hemi-slant submanifolds of a locally product
Riemannian manifold in detail.

\section{Preliminaries}
This section is devoted to preliminaries. Actually, in subsection
2.1 we present the basic background needed for a locally product
Riemannian manifold. Theory of submanifolds and distributions related to
the study are given in subsection 2.2.

\subsection{Locally product Riemannian manifolds}
Let $\bar{M}$ be an $m$-dimensional manifold with a tensor field of
type (1,1) such that
\begin{equation}
\label{e00}
\begin{array}{c}
F^{2}=I, (F\neq\pm I)
\end{array},
\end{equation}
where $I$ is the identity morphism on the tangent bundle $T\bar{M}$
of $\bar{M}$. Then we say that $\bar{M}$ is an \emph{almost product
manifold} with almost product structure $F.$ If an almost product
manifold $(\bar{M},F)$ admits a Riemannian metric $g$ such that
\begin{equation}
\label{e0}
\begin{array}{c}
g(F\bar{U},F\bar{V})=g(\bar{U},\bar{V})
\end{array}
\end{equation}
for all $\bar{U},\bar{V}\in T\bar{M},$ then $\bar{M}$ is called an
\emph{almost
product Riemannian manifold}.\\

Next, we denote by $\overline{\nabla}$ the Riemannian connection
with respect to $g$ on $\bar{M}$. We say that $\bar{M}$ is a
\emph{locally product Riemannian manifold}, (briefly, \emph{l.p.R.
manifold}) if we have
\begin{eqnarray}
\label{e1}
(\overline{\nabla}_{\bar{U}}\,\,F)\bar{V}=0\,,
\end{eqnarray}
for all $\bar{U}, \bar{V}\in T\bar{M}$ \cite{Yan}.

\subsection{Submanifolds}
Let $M$ be a submanifold of a l.p.R. manifold $(\bar{M},g,F)$. Let  $\overline\nabla,
{\nabla},$ and $\nabla^{\bot}$ be the Riemannian,
induced Riemannian, and induced normal connection in $\bar{M}, M$ and the
normal bundle $T^{\bot}M$ of $M$, respectively. Then for all  $U,V\in TM$ and $\xi\in T^{\bot}M$
the Gauss and Weingarten formulas are given by
\begin{equation}
\label{e2}
\begin{array}{c}
\overline{\nabla}_{U}V={\nabla}_{U}V+h(U,V)
\end{array}
\end{equation}
and
\begin{equation}
\label{e3}
\begin{array}{c}
\overline\nabla_{U}\xi=-A_{\xi}U+\nabla_{U}^{\bot}\xi
\end{array}
\end{equation}
where $h$ is the second fundamental form related to shape operator. $A$ corresponding to the
normal vector field $\xi$ is given by
\begin{equation}
\label{e4}
\begin{array}{c}
g(h(U,V),\xi)=g(A_{\xi}U,V)
\end{array}.
\end{equation}

A submanifold $M$ is said to be \emph{totally geodesic} if its
second fundamental form vanishes identically, that is, $h=0,$ or
equivalently $A_{\xi}=0.$ We say that $M$ is \emph{totally
umbilical} submanifold in $\overline{M}$ if for all $U,V\in TM$ we
have
\begin{equation}
\label{e5}
\begin{array}{c}
h(U,V)=g(U,V)H
\end{array},
\end{equation}
where $H$ is the mean curvature vector field of $M$ in $\bar{M}$. A
normal vector field $\xi$ is said to be parallel, if
$\nabla^{\perp} _{U}\xi =0$ for each vector field $U \in TM.$\\

The Riemannian curvature tensor $ {\overline R}  $ of $ {\bar M}$ is given by
\begin{eqnarray} \label{}
{\overline R} (\bar U, \bar V)= \big [ {\overline \nabla}_{\bar U},
{\overline \nabla}_{\bar V}\big ]-{\overline\nabla}_{[\bar U,\bar
V]} ,
\end{eqnarray}
where $\bar U, \bar V \in T \bar M $

\noindent Then the Codazzi equation is given by
\begin{eqnarray} \label{}
\big ({\overline R}(U,V) W \big )^{\perp}= ({\overline \nabla}_{U} h )(V, W)- ({\overline \nabla}_{V} h )(U, W)
\end{eqnarray}
for all $U\,\, V,\,W \in TM$. Here, $ {\perp}$ denotes the normal component and the covariant derivative of $h$, denoted by
${\overline \nabla}_{U} h$ is defined by
\begin{eqnarray} \label{a1}
({\overline \nabla}_{U} h )(V, W)= {\nabla}^{\perp} _{U} h(V,W)-h({\nabla}_{U} V,W)-h(V,{\nabla}_{U} W).
\end{eqnarray}
Now, we write
\begin{equation} \label{e6}
\begin{array}{c}
FU=TU+NU
\end{array},
\end{equation}
for any $U\in TM$. Here $TU$ is the tangential part of $FU,$ and
$NU$ is the normal part of $FU.$ Similarly, for any $\xi\in
T^{\bot}M$, we put
\begin{equation}
\label{e7}
\begin{array}{c}
F\xi=t\xi+\omega\xi
\end{array},
\end{equation}
where $t\xi$ is the tangential part of $F\xi,$ and $\omega\xi$ is the
normal part of $F\xi.$\\

A distribution $\mathcal{D}$ on a manifold $\bar{M}$ is called
\emph{autoparallel} if $\overline\nabla_{X}Y\in\mathcal{D}$ for any
$X,Y\in\mathcal{D}$ and called \emph{parallel} if
$\overline \nabla_{U}X\in\mathcal{D}$ for any $X\in\mathcal{D}$ and $U\in TM.$
If a distribution $\mathcal{D}$ on $\bar{M}$ is autoparallel, then
it is clearly integrable, and by Gauss formula $\mathcal{D}$ is
totally geodesic in $\bar{M}$. If $\mathcal{D}$ is parallel then the
orthogonal complementary distribution $\mathcal{D}^{\perp}$ is also
parallel, which implies that $\mathcal{D}$ is parallel if and only
if $\mathcal{D}^{\perp}$ is parallel. In this case $\bar{M}$ is
locally product of the leaves of $\mathcal{D}$ and
$\mathcal{D}^{\perp}$. Let $M$ be a submanifold of $\bar{M}$. For
two distributions $\mathcal{D}_{1}$ and $\mathcal{D}_{2}$ on $M$, we
say that $M$ is $(\mathcal{D}_{1},\mathcal{D}_{2})$ mixed totally
geodesic if for all $X\in\mathcal{D}_{1}$ and $Y\in\mathcal{D}_{2}$
we have $h(X,Y)=0,$ where $h$ is the second fundamental form of $M$
\cite{Muku, Yan}.
\section{Hemi-slant submanifolds of a\\
locally product Riemannian manifold}
In this section, we define the notion of hemi-slant submanifold and
observe its effect to the tangent bundle of the submanifold and
canonical projection operators and start to study hemi-slant
submanifolds of a locally product Riemannian manifold.\\

Let $(\bar{M},g,F)$ be a locally product Riemannian manifold and let
$M$ be a submanifold of $\bar{M}$. A distribution $\mathcal{D}$ on
$M$ is said to be a \emph{slant distribution} if for
$X\in\mathcal{D}_{p},$ the angle $\theta$ between $FX$ and
$\mathcal{D}_{p}$ is constant, i.e., independent of $p\in M$ and
$X\in\mathcal{D}_{p}.$ The constant angle $\theta$ is called the
slant angle of the slant distribution $\mathcal{D}$ . A submanifold
$M$ of $\bar{M}$ is said to be a \emph{slant submanifold} if the
tangent bundle $TM$ of $M$ is slant \cite{Li, S}. Thus, the
$F-$invariant and $F-$anti-invariant submanifolds are slant
submanifolds with slant angle $\theta=0$ and $\theta=\pi / 2$,
respectively. A slant submanifold which is neither $F-$invariant nor
$F-$anti-invariant is called a \emph{proper} slant submanifold.
\begin{definition}
A \emph{hemi-slant submanifold} $M$ of a locally product Riemannian
manifold $\bar{M}$ is a submanifold which admits two orthogonal
complementary distributions $\mathcal{D}^{\perp}$ and
$\mathcal{D}^{\theta}$ such that
\end{definition}
\textbf{(a)} $TM$ admits the orthogonal direct decomposition
$TM=\mathcal{D}^{\perp}\oplus\mathcal{D}^{\theta}$

\textbf{(b)} The distribution $\mathcal{D}^{\perp}$ is
${F-}$anti-invariant, i.e., $F\mathcal{D}^{\perp}\subseteq
T^{\bot}M.$

\textbf{(c)} The distribution $\mathcal{D}^{\theta}$ is slant with slant angle $\theta$.\\

In this case, we call $\theta$ the slant angle of $M$. Suppose the
dimension of distribution $\mathcal{D}^{\perp}$ (resp.
$\mathcal{D}^{\theta}$ ) is $p$ (resp. $q$ ). Then we easily see
that the following particular cases.\\

\textbf{(d)} If $q=0,$ then $M$ is an anti-invariant submanifold
\cite{A}.

\textbf{(e)} If $p=0$ and $\theta=0$, then $M$ is an invariant
submanifold \cite{A}.

\textbf{(f)} If $p=0$ and $\theta\neq0, \frac{\pi}{2}$, then $M$ is
a proper slant submanifold \cite{S}.

\textbf{(g)} If $\theta=\frac{\pi}{2}$, then $M$ is an
anti-invariant submanifold.

\textbf{(h)} If $p\neq0$ and $\theta=0$, then $M$ is a
semi-invariant submanifold \cite{Be}.
\\

We say that the hemi-slant submanifold $M$ is \emph{proper} if
$p\neq0$ and $\theta\neq0,\frac{\pi}{2}.$

\begin{lemma} Let $M$ be a proper hemi-slant submanifold of a l.p.R. manifold $\bar{M}.$
Then we have,
\begin{equation}\label{e11}
\begin{array}{c}
 F(\mathcal{D}^{\perp})\perp N(\mathcal{D}^{\theta})
\end{array}.
\end{equation}
\end{lemma}

\begin{proof}  For any  $X\in\mathcal{D}^{\perp}$ and $Z\in \mathcal{D}^{\theta}$, using  (2.2)
and (2.11), we have \\ $g(FX,NZ)=g(FX,FZ)=g(X,Z)=0.$ This completes
the proof.
\end{proof}
In view of Lemma 3.2, for a hemi-slant submanifold $M$ of a l.p.R.
manifold $\bar{M},$  the normal bundle $T^{\bot}M$ of $M$ is
decomposed as
\begin{equation} \label{e12}
\begin{array}{c}
T^{\bot}M=F(\mathcal{D}^{\perp})\oplus N(\mathcal{D}^{\theta})
\oplus\mu
\end{array},
\end{equation}
where $\mu$ is the orthogonal complementary distribution of $
F(\mathcal{D}^{\perp})\oplus N(\mathcal{D}^{\theta})$ in $T^{\bot}M$
and it is invariant subbundle of $T^{\bot}M$ with
respect to $F.$\\

The following facts follow easily from (2.1), (2.11) and (2.12) and  will be
used later.
\begin{eqnarray}\label{e131}
&&(a)\quad T^{2}+tN=I, \quad\quad (b)\quad \omega^{2}+Nt=I,\nonumber\\
&&(c)\quad NT+\omega N=0,~~~~\quad(d)\quad Tt+t\omega=0.
\end{eqnarray}
As in a slant submanifold \cite{S}, for a hemi-slant submanifold $M$
of a l.p.R. manifold $\overline{M}$ , we have
\begin{equation}\label{e141}
\begin{array}{c}
T^{2}Z=\cos^{2}\!\theta Z\,,
\end{array}
\end{equation}
\begin{equation}\label{e151}
\begin{array}{c}
g(TZ,TW)=\cos^{2}\!\theta g(Z,W)
\end{array}
\end{equation}
and
\begin{equation}
\begin{array}{c}
g(NZ,NW)=\sin^{2}\!\theta g(Z,W)
\end{array},
\end{equation}
where $Z,W\in \mathcal{D}^{\theta}\,.$\\
\begin{lemma} Let $M$ be a proper hemi-slant submanifold of a l.p.R. manifold $\bar{M}.$
Then we have,
\begin{equation} \label{e8}
\begin{array}{c}
(a)\quad T(\mathcal{D}^{\perp})=\{0\}, \quad\quad (b)\quad
T(\mathcal{D}^{\theta})=\mathcal{D}^{\theta}
\end{array}.
\end{equation}
\end{lemma}
\begin{proof} Since $\mathcal{D}^{\perp}$ is anti-invariant with respect to $F$,
(a) follows from (2.11). For any  $Z\in \mathcal{D}^{\theta}$ and
$X\in \mathcal{D}^{\perp}$, using (2.1), (2.2) and (2.11), we have
$g(TZ,X)=g(FZ,X)=g(Z,FX)=0.$ Hence, we conclude that
$T(\mathcal{D}^{\theta})\perp\mathcal{D}^{\perp}$. Since
$T(\mathcal{D}^{\theta})\subseteq TM$, it follows that
$T(\mathcal{D}^{\theta})\subseteq\mathcal{D}^{\theta}.$ Let $W$ be
in $ \mathcal{D}^{\theta}$. Then using (3.4), we have
$W=\frac{1}{\cos^2\!\theta}(\cos^{2}\!\theta W)=\frac{1}{\cos^2
\!\theta}T^{2}W=\frac{1}{\cos^2\!\theta}T(TW)$. So, we find $W \in
T(\mathcal{D}^{\theta})$. It follows that $\mathcal{D}^{\theta}
\subseteq T(\mathcal{D}^{\theta}).$ Thus, we get the assertion (b).
\end{proof}
Thanks to Theorem 3.1 \cite{S}, we characterize hemi-slant
submanifolds of a l.p.R. manifold.
\begin{theorem}
Let $M$ be a submanifold of a l.p.R. manifold $\bar{M}$. Then $M$ is
a hemi-slant submanifold if and only if there exists a constant
$\lambda \in [0,1]$ and a distribution $\mathcal{D}$ on $M$ such
that
\end{theorem}

\textbf{(a)} \, $\mathcal{D}= \{ U \in TM \,\,\,\vert \,\,\,T^2
U=\lambda U\}$,

\textbf{(b)} \, for any $X \in TM$ orthogonal to $\mathcal{D}$,
$TX=0$.

Moreover, in this case $\lambda=\cos^2\!\theta$, where $\theta$ is
the slant angle of $M$.
\begin{proof} Let $M$ be a hemi-slant submanifold of $\bar{M}$. By the definition of
hemi-slant submanifold, we have $\mathcal{D}= \mathcal{D}^{\theta} $
and $\lambda=\cos^2\!\theta$. So, $(a)$ follows. \noindent  $(b)$
follows from Lemma 3.3. Conversely,  $(a)$ and  $(b)$ imply
$TM=\mathcal{D}^{\perp}\oplus\mathcal{D}$. Since
$T(\mathcal{D})\subseteq\mathcal{D}$, we conclude that
$\mathcal{D}^{\perp}$ is an anti-invariant distribution from $(b)$.
\end{proof}
\noindent {\bf Example.} Consider the Euclidean 6-space
$\mathbb{R}^6$ with usual metric $g$. Define the almost product
structure $F$ on $(\mathbb{R}^6, g)$ by
\begin{equation}\nonumber\\
\begin{array}{c}
F(\displaystyle\frac{\partial}{\partial x_{i}})=\frac{\partial}{\partial y_{i}}, \quad F(\displaystyle\frac{\partial}{\partial y_{i}})=\frac{\partial}{\partial x_{i}},
\quad i=1,2,3.
\end{array}
\end{equation}
Where $(x_1, x_2,x_{3},y_1, y_2,y_{3})$ are natural coordinates of $\mathbb{R}^6$. Then $\bar{M}= (\mathbb{R}^6, g,F)$
be an almost product Riemannian manifold. Furthermore,
it is easy to see that $\bar M$ is a l.p.R. manifold. Let M be a submanifold of $\bar{M}$ defined by
\begin{equation}\nonumber
f(u,v,w)= \big( {u\over \sqrt 2}, {u\over \sqrt 2},u+v,{w\over \sqrt 2}, {w\over \sqrt 2}, 0 \big)\,, \quad \quad {u\neq 0}.
\end{equation}\\
Then, a local frame of $TM$ is given by
\begin{eqnarray}
&&X=\frac{\partial}{\partial x_{3}}\,,\qquad\qquad\qquad\qquad\qquad\qquad\nonumber\\[.2cm]
&&Z={1\over \sqrt 2}\frac{\partial}{\partial x_{1}}+ {1\over \sqrt 2}\frac{\partial}{\partial x_{2}}+\frac{\partial}{\partial x_{3}},\qquad\qquad\qquad\qquad\qquad\qquad\nonumber\\[.4cm]
&&W={1\over \sqrt 2}\frac{\partial}{\partial y_{1}} + {1\over \sqrt 2}\frac{\partial}{\partial y_{2}}.\qquad\qquad\qquad\qquad\qquad\qquad\nonumber
\end{eqnarray}\\
By using the almost product structure $F$ above, we see that $FX$ is
orthogonal to $TM$, thus $\mathcal{D}^{\perp}= \rm{span} \{X\}$.
Moreover, it is not difficult to see that $\mathcal{D}^{\theta}= span\{
Z, W \}$ is a slant distribution with slant angle $\theta = {\pi/
3}$\,. Thus, $M$ is a proper hemi-slant submanifold of $\bar M$.
\section{Integrability}
In this section, we give a necessary and sufficient condition for
the integrability of the slant distribution of the hemi- slant
submanifold. After that we prove that the anti invariant
distribution of the hemi-slant submanifold is always integrable and
give some applications of this result.\\

Let $M$ be a submanifold of a l.p.R. manifold $\bar{M}$. For any
$U$,$V\in TM$, we have $\overline{\nabla}_{U} FV=
F\overline{\nabla}_{U}V$ from (\ref{e1}). Then, using
(\ref{e2}-\ref{e3}), (\ref{e6}-\ref{e7}) and identifying the
components from $TM$ and  $T^{\perp} M$, we have the following.
\begin{lemma} Let $M$ be a submanifold of a l.p.R. manifold $\bar{M}.$
Then we have,
\begin{equation}
\label{e17}\begin{array}{c} {\nabla}_{U} TV- A_{NV} U=
T{\nabla}_{U}V+ t\,h(U,V)\,,
\end{array}
\end{equation}
\end{lemma}
\begin{equation}\label{eon7}
\begin{array}{c}
h(U,TV)+ {\nabla}_{U}^{\perp} NV = N{\nabla}_{U}V+ \omega\,h(U,V)
\end{array}.
\end{equation}
for all  $U$,$V\in TM$.

In a similar way, we have that:
\begin{lemma} Let $M$ be a submanifold of a l.p.R. manifold $\overline{M}.$
Then we have,
\begin{equation} \label{e18}
\begin{array}{c}
{\nabla}_{U}\,t\,{\xi}- A_{\omega{\xi}}U= -TA_{\xi} U + t {{\nabla}
^{\perp} _{U}}\,{\xi}
\end{array},
\end{equation}
\end{lemma}
\begin{equation}
\begin{array}{c}
h(U,t\,{\xi})+ {\nabla}_{U}^{\perp}\,\omega\,{\xi}= -N A_{\xi} U+ \omega{{\nabla} ^{\perp} _{U}}\,{\xi}
\end{array}
\end{equation}
for any $U \in TM$ and ${\xi} \in T ^{\perp} M$.
\begin{theorem} Let $M$ be a hemi-slant manifold of a l.p.R. manifold $\overline{M}.$
Then, the slant distribution $\mathcal{D}^{\theta}$ is integrable if and only if
\begin{equation}
\begin{array}{c}
A_{NZ} W- A_{NW} Z + {\nabla}_{Z} TW- {\nabla}_{W} TZ \in \mathcal{D}^{\theta}
\end{array}
\end{equation}
\end{theorem}
for any $Z$,$W \in \mathcal{D}^{\theta}$.
\begin{proof} From (\ref{e17}), we have
\begin{equation}\label{e22}
\begin{array}{c}
{\nabla}_{Z} TW- A_{NW} Z= T {\nabla}_{Z} W +t\,h(Z,V)
\end{array}
\end{equation}
and
\begin{equation}\label{e23}
\begin{array}{c}
{\nabla}_{W} TZ- A_{NZ} W= T {\nabla}_{W} Z +t h(W,Z)
\end{array}
\end{equation}
for any $Z$, $W\in \mathcal{D}^{\theta}$. Since $ h$ is a symmetric $(0,2)$-type tensor field, from (\ref{e22}) and (\ref{e23}), we get
\begin{equation}\label{e24}
\begin{array}{c}
A_{NZ} W- A_{NW} Z + {\nabla}_{Z} TW- {\nabla}_{W} TZ= T [Z,W]
\end{array}.
\end{equation}
Thus, our assertion follows from  (3.7-$b$) and (\ref{e24}).
\end{proof}
\noindent The following we give an application of Theorem 4.3.
\begin{theorem}
Let $M$ be a hemi-slant manifold of a l.p.R. manifold $\bar{M}$. If $M$ is
$\mathcal{D}^{\theta}$-totally geodesic, then the slant distribution $\mathcal{D}^{\theta}$
is integrable.
\end{theorem}
\begin{proof}Suppose that $M$ is  $\mathcal{D}^{\theta}$-totally geodesic, that is, for any $Z$, $W \in \mathcal{D}^{\theta}$ we have
\begin{equation}\label{e43}
\begin{array}{c}
 h(Z,W)=0.
\end{array}
\end{equation}
Thus, from (\ref{e17}), using (\ref{e43}), we have
 \begin{equation}\label{e44}
\begin{array}{c}
A_{NZ} W-\nabla_W TZ =- T\nabla_W Z
\end{array}
\end{equation}
 and similarly
\begin{equation}\label{e45}
\begin{array}{c}
 A_{NW} Z-\nabla_Z TW =- T\nabla_Z W\,.
\end{array}
\end{equation}
From (\ref{e44}) and (\ref{e45}), using Lemma 3.3, we get
\begin{equation}\label{e46}
\begin{array}{c}
g(A_{NZ} W-A_{NW}Z + \nabla_Z TW -\nabla_W TZ, X) = g (T [Z,W],X)=0
\end{array}
\end{equation}
for any $X \in \mathcal{D}^{\perp}$. The last equation (\ref{e46}) says that
\begin{equation}\nonumber
\begin{array}{c}
A_{NZ} W-A_{NW} Z+ \nabla_Z TW -\nabla_W TZ \in \mathcal{D}^{\theta}
\end{array}
\end{equation}
and by Theorem 4.3, we deduce that $\mathcal{D}^{\theta}$ is integrable.
%Let $h^{\theta}$ (resp. $\overline{h})$ be the second fundamental form of the immersion of $M^ {\theta}$ in $M$ (resp. $\overline{M}$)\,.
\end{proof}
\begin{lemma} Let $M$ be a hemi-slant submanifold of a l.p.R. manifold $\bar{M}$. Then,
\begin{equation}\label{e25}
\begin{array}{c}
A_{NX} Y= -A_{NY} X
\end{array}
\end{equation}
\end{lemma}
for any $X$,$Y \in \mathcal{D}^{\perp}$.
\begin{proof} For any $X \in \mathcal{D}^{\perp}$ and $U\in TM$, using (3.7-$a$), we have
\begin{equation}\label{e26}
\begin{array}{c}
 -T{\nabla}_{U} X= A_{NX} U+ t\,h(U,X)
 \end{array}
\end{equation}
from (\ref{e17}). Let $Y$ be in $\mathcal{D}^{\perp}$. Using
(3.7-$b$), we obtain
\begin{equation}\label{e27}
\begin{array}{c}
0=-g(T{\nabla}_{U} X, Y) =g( A_{NX} U, Y)+ g(t h(U,X),Y)\,
\end{array}
\end{equation}
from (\ref{e26}). On the other hand, using (\ref{e0}), (\ref{e4}), (\ref{e6}) and (\ref{e7}), we find
\begin{equation}\label{e28}
\begin{array}{c}
g(t\, h(U,X),Y)=g( A_{NY} U, X).
\end{array}
\end{equation}
Thus, from (\ref{e27}) and (\ref{e28}), we deduce that
\begin{equation}
\begin{array}{c}
g( A_{NX} Y+ A_{NY} X, U)=0.
\end{array}
\end{equation}
This equation gives (\ref{e25}).
\end{proof}
\begin{theorem} Let $M$ be a hemi-slant submanifold of a l.p.R. manifold $\bar{M}$.
Then the anti-invariant distribution $\mathcal{D}^{\perp}$ is integrable if and only if
\begin{equation}
\begin{array}{c}
A_{NX} Y= A_{NY} X
\end{array}
\end{equation}
\end{theorem}
for all $X$, $Y \in \mathcal{D}^{\perp}$.
\begin{proof} From (\ref{e17}), using (3.7-$a$), we have
\begin{equation}\label{e31}
\begin{array}{c}
-A_{NY} X= T {\nabla}_{X} Y +t\,h(X,Y)
\end{array}
\end{equation}
 for all $X \in \mathcal{D}^{\perp}$. By interchanging $X$ and $Y$ in (\ref{e31}),
 then subtracting it from (\ref{e31}) we obtain
\begin{equation}\label{e32}
\begin{array}{c}
A_{NX} Y- A_{NY} X=T[X,Y]
\end{array}.
\end{equation}
Because of (3.7-$a$), we know that $\mathcal{D}^{\perp}$ is
integrable if and only if $T[X,Y]=0$ for all $X$,$Y \in
\mathcal{D}^{\perp}$. So, our assertion comes from (\ref{e32}).
\end{proof}
\noindent By Lemma 4.5 and Theorem 4.6, we have the following
result.
\begin{cor} Let $M$ be a hemi-slant submanifold of a l.p.R. manifold $\overline{M}$.
Then the anti-invariant distribution $\mathcal{D}^{\perp}$ is integrable if and only if
\begin{equation}\label{e33}
\begin{array}{c}
 A_{NX} Y= 0
 \end{array}
\end{equation}
\end{cor}

for all $X$, $Y \in \mathcal{D}^{\perp}$.\\

\noindent Now, we give main result of this section.
\begin{theorem} Let $M$ be a hemi-slant submanifold of a l.p.R. manifold $\bar{M}$.
Then the anti-invariant distribution $\mathcal{D}^{\perp}$ is always integrable.
\end{theorem}
\begin{proof} Let $\bar{M}$ be a l.p.R. manifold with Riemannian metric $g$ and almost product structure $F$.
Define the symmetric (0,2)-type tensor field $\Omega$ by $\Omega
(\bar{U},\bar{V})= g(F\bar{U},\bar{V})$ on the tangent bundle $T
\bar{M}$. It is not difficult to see that $(\nabla_{\bar{U}} \Omega)
(\bar{V},\bar{W})=g((\nabla_{\bar{U}}F )\bar{V},\bar{W}) $ on $T
\bar{M}$. Thus, because  of (\ref{e1}), we deduce that
\begin{equation}\nonumber
\begin{array}{c}
3 \,d \Omega(\bar{V},\bar{W},\bar{U})= {\mathcal{G}}
(\nabla_{\bar{U}} \Omega) (\bar{V},\bar{W})=0
\end{array}
\end{equation}
for all $\bar{U}, \bar{V}, \bar{W} \in T \bar{M}$, that is,  $
d\Omega \equiv 0$\,, where ${\mathcal{G}}$ denotes the cyclic sum
over $\bar U, \bar V, \bar W \in T{\bar M}$.
\noindent Next, for any $X$, $Y \in \mathcal{D}^{\perp}$ and $U\in TM$
we have
\begin{eqnarray}\nonumber
&&0= 3\, d \Omega(U,X,Y)= U\, \Omega(X,Y)+ X\, \Omega(Y,U)+ Y\, \Omega(U,X)\nonumber\\[.2cm]
&&\qquad\qquad\qquad\qquad-\Omega([U,X],Y)-\Omega([X,Y],U)-\Omega([Y,U],X)\nonumber\\[.2cm]
&&\qquad\qquad\qquad\qquad=g(T[Y,X],U])\,.\nonumber
\end{eqnarray}
It follows that $T[X,Y]= 0$ and because of (3.7-$a$), $[Y,X] \in
\mathcal{D}^{\perp}$\,.
\end{proof}
We remark that we used Tripathi's technique [8] in the proof above.
\begin{cor} Let $M$ be a hemi-slant submanifold of a l.p.R. manifold $\bar{M}$. Then the following facts hold:
\begin{eqnarray}
&A_{{N{D^{\perp}}}} {D^{\perp}}=0 \label{e304}\\[.2cm]
&A_{NX} Z\in {D^{\theta}},\,\,\quad \text{i.e.,}\,\,\, A_{{N{D^{\perp}}}} {D^{\theta}} \subseteq {D^{\theta}} \label{e305}
\end{eqnarray}
\rm{and}
\begin{equation}
g(h(TM,{\mathcal{D}^{\perp}}), N{\mathcal{D}^{\perp}})=0\,,\label{e306}
\end{equation}
\end{cor}
where $X\in \mathcal{D}^{\perp}$ and $Z\in\mathcal{D}^{\theta}$.
\begin{proof} (\ref{e304}) follows from Corollary 4.7 and Theorem 4.8. (\ref{e305}) follows from (\ref{e304}).
Finally, using (\ref{e4}), (\ref{e304}) gives (\ref{e306}).
\end{proof}
Next, we give another application of Theorem 4.8.
\begin{theorem} Let $M$ be a proper hemi-slant submanifold of a l.p.R. manifold $\bar{M}$.
The anti-invariant distribution $\mathcal{D}^{\perp}$ defines a totally geodesic foliation on $M$ if and ony if
$h(\mathcal{D}^{\perp},\mathcal{D}^{\perp}) {\perp} \,\,N {\mathcal{D}^{\theta}}$.
\end{theorem}
\begin{proof} For $X,Y\in  \mathcal{D}^{\perp}$, we put $\nabla_X Y =\, ^{\perp}\nabla_X Y +\, ^{\theta}\nabla_X Y$,
where $^{\perp} \nabla_X Y$ (resp. $^{\theta}\nabla_X Y$) denotes
the anti-invariant (resp. slant) part of $\nabla_X Y$. Then using
Lemma 3.3 and (3.5), for any $Z\in\mathcal{D}^{\theta}$  we have
\begin{equation}\label{e139}
\begin{array}{c}
g(\nabla_X Y,Z)= g(^\theta\nabla_X Y, Z)= {1\over
{\cos^2\!\theta}}\, g( T ^\theta{\nabla_X} Y, TZ)= {1\over
\cos^2\!\theta}\, g( T {\nabla_X} Y, TZ)\,.
\end{array}
\end{equation}
On the other hand, from (4.1), we have
\begin{equation}\label{e38}
\begin{array}{c}
T \nabla_X Y +t\, h(X,Y)=-A_{NY} X= 0
\end{array},
\end{equation}
since the distribution $\mathcal{D}^{\perp}$ is integrable. So,
using (4.26), from (\ref{e139}), we get
\begin{equation}\label{40}
\begin{array}{c}
g(\nabla_X Y,Z)= -{1\over {\cos^2\!\theta}}\, g(t\, h(X,Y),TZ )=
-{1\over {\cos^2\!\theta}}\, g(F h(X,Y),TZ )\,.
\end{array}
\end{equation}
Here, using (\ref{e0}), (\ref{e6}) and (\ref{e141}), we find
\begin{equation}\label{41}
\begin{array}{c}
g(F\,h(X,Y),TZ) = g(h(X,Y),NTZ ).
\end{array}
\end{equation}
From (\ref{40}) and (\ref{41}), we get
\begin{equation}\label{e400}
\begin{array}{c}
g(\nabla_X Y,Z)= -{1\over {\cos^2\!\theta}}\, g( h(X,Y),NTZ )\,.
\end{array}
\end{equation}
Since $TZ\in \mathcal{D}^{\theta}$, our assertion comes from
(\ref{e400}).
\end{proof}
\section{Hemi-slant product}
In this section, we give a necessary and sufficient condition for a
proper hemi-slant submanifold to be a hemi-slant product.
\begin{definition} A proper hemi-slant submanifold $M$ of a l.p.R. manifold $\bar{M}$ is
called a hemi-slant product if it is locally product Riemannian of
an anti-invariant submanifold $M_ {\perp}$ and a proper slant
submanifold $M_ {\theta}$ of $\bar{M}$.

Now, we are going to examine the problem when a proper hemi-slant
submanifold of a l.p.R. manifold is a hemi-slant product?
\end{definition}
We first give a result which is equivalent to Theorem 4.10.
\begin{theorem} Let $M$ be a proper hemi-slant submanifold of a l.p.R. manifold $\bar{M}$.
Then the anti-invariant $\mathcal{D}^{\perp}$ defines a totally geodesic foliation on $M$ if and only if
\begin{equation} \label{e47}
\begin{array}{c}
g(A_{NY} Z, X)= -g(A_{NZ} Y,X),
\end{array}
\end{equation}
where $X$, $Y \in \mathcal{D}^{\perp}$ and $ Z \in \mathcal{D}^{\theta}$.
\end{theorem}
\begin{proof} For any  $X$, $Y \in \mathcal{D}^{\perp}$ and $Z\in \mathcal{D}^{\theta}$,
using (\ref{e2}), (\ref{e0}), and (\ref{e1}), we have
\begin{equation}\nonumber \label{}
\begin{array}{c}
g(\nabla_X Y,Z)=g(\overline{\nabla}_X Y,Z)=g(\overline {\nabla}_X FY,FZ).
\end{array}
\end{equation}
Hence, using (\ref{e6}), (\ref{e2}), (\ref{e3}) and (\ref{e0}), we
obtain
\begin{equation}\nonumber \label{}
\begin{array}{c}
g(\nabla_X Y,Z)=-g(A_{NY} X, TZ)+g({\nabla}_X Y,FNZ)+ g(h(X,Y),FNZ).
\end{array}
\end{equation}
Here, using (3.3)-c, (3.3)-a, (2.12) and (3.4), we have\\

$FNZ=tNZ-NTZ$ and $tNZ=Z-T^2 Z=\sin^2\!{\theta} Z$. Thus, with the
help of (2.6), we get
\begin{eqnarray}\nonumber \label{}
g(\nabla_X Y,Z)=-g(A_{NY}X,TZ)+ \sin^2\!{\theta}g({\nabla}_X
Y,Z)-g(A_{NTZ} Y,X).
\end{eqnarray}
After some calculations, we find
\begin{equation}\nonumber \label{}
\begin{array}{c}
\cos^2\!{\theta}g(\nabla_X Y,Z)=-g(A_{NY} TZ, X)-g(A_{NTZ} Y,X).
\end{array}
\end{equation}
It follows that the distribution $\mathcal{D}^{\perp}$ defines a totally geodesic foliation on $M$ if and only if
\begin{equation} \label{e48}
\begin{array}{c}
g(A_{NY} TZ, X)= -g(A_{NTZ} Y,X).
\end{array}
\end{equation}
Putting $Z=TZ$ in (\ref{e48}), we obtain (\ref{e47}) and vice versa.
\end{proof}
\begin{theorem} Let $M$ be a proper hemi-slant submanifold of a l.p.R. manifold $\bar{M}$.
Then the distribution $\mathcal{D}^{\theta}$ defines a totally geodesic foliation on $M$ if and only if
\begin{equation} \label{e49}
\begin{array}{c}
g(A_{NX} W, Z)= -g(A_{NW} X,Z),
\end{array}
\end{equation}
where $X$, $Y \in \mathcal{D}^{\perp}$ and $ Z, W \in \mathcal{D}^{\theta}$.
\end{theorem}
\begin{proof} Using (\ref{e2}), (\ref{e0}), and (\ref{e1}), we have $g(\nabla_Z W,X)=g(\overline \nabla_Z {FW}, FX)$
for any $ Z, W \in \mathcal{D}^{\theta}$ and $X \in
\mathcal{D}^{\perp}$. Next, using (\ref{e6}) and (\ref{e11}),
obtain $g(\nabla_Z W,X)=-g(TW, \overline \nabla_Z NX)-g(NW,\overline
\nabla_Z FX)$.
Hence, using (\ref{e3}) and (\ref{e00}), we get $g(\nabla_Z W,X)= g(TW,A_{NX} Z)-g(FNW, \overline\nabla_Z X)$.
With the help of (2.12), (3.3)-(a), (3.3)-(c) and (\ref{e2}),
we arrive at
\begin{equation}\nonumber \label{}
\begin{array}{c}
g(\nabla_Z W,X)=-g(A_{NX} Z, TW)- \sin^2\!{\theta}\, g({\nabla}_Z
X,W)+g(h(X,Z), NTW).
\end{array}
\end{equation}
Upon direct calculation, we find
\begin{equation}\nonumber \label{}
\begin{array}{c}
 \cos^2\!{\theta} \,\,g({\nabla}_Z W,X)=g(A_{NX} TW,Z)+ g(A_{NTW} X,Z)
\end{array}
\end{equation}
So, we deduce that the slant distribution $\mathcal{D}^{\theta}$ defines a totally geodesic foliation if and only if
\begin{equation} \label{e50}
\begin{array}{c}
g(A_{NX} TW, Z)= -g(A_{NTW} X,Z),
\end{array}
\end{equation}
By putting $W=TW$, we see that the last equation is equivalent to the equation (\ref{e49}).
\end{proof}
Thus, from Theorems 5.2 and 5.3, we obtain the expected result.
\begin{cor}  Let $M$ be a proper hemi-slant submanifold of a l.p.R. manifold $\bar{M}$.
Then $M$ is a hemi-slant product manifold $M=M_{\perp} \times
M_{\theta} $ if and only if
\begin{equation} \label{e51}
\begin{array}{c}
A_{NX} Z=-A_{NZ} X,
\end{array}
\end{equation}
where  $X \in \mathcal{D}^{\perp}$ and $ Z \in \mathcal{D}^{\theta}$.
\end{cor}

\section{Hemi-slant submanifolds with parallel canonical structures}
In this section, we get several results for the hemi-slant submanifolds
with parallel canonical structures usingthe previous results.\\

Let $M$ be any submanifold of a l.p.R. manifold $\bar{M}$ with the
endomorphisim $T$ and the normal bundle valued 1-form $N$  defined
by (\ref{e6}). We put
\begin{equation} \label{e52}
\begin{array}{c}
(\overline \nabla_UT)V= \nabla_UTV-T \nabla_UV
\end{array}
\end{equation}
and
\begin{equation} \label{e53}
\begin{array}{c}
(\overline \nabla_UN) V= \nabla^{\perp} _UNV -N\nabla_UV
\end{array}
\end{equation}
for any $U$,$V\in TM$. Then the endomorphisim $T$ (resp.1-form N) is
parallel if $\overline \nabla T \equiv 0$ \,\, (resp. $\overline
\nabla N \equiv 0)$\,. From (\ref{e17}) and (\ref{eon7}) we have
\begin{equation}\label{e54}
\begin{array}{c}
({\overline \nabla}_{U} T) V= A_{NV} U+ t\,h(U,V)\,
\end{array}
\end{equation}
and
\begin{equation}\label{e55}
\begin{array}{c}
({\overline \nabla}_{U} N) V= \omega\,h(U,V) -h(U,TV),
\end{array}
\end{equation}
respectively.
\begin{theorem}
Let $M$ be any submanifold of a l.p.R. manifold $\bar{M}$. Then $T$
is parallel, i.e., $\overline \nabla T \equiv 0$  if and only if
\begin{equation} \label{e56}
\begin{array}{c}
A_{NV} U=- A_{NU} V,
\end{array}
\end{equation}
for all $U$,$V\in TM$.
\end{theorem}
\begin{proof} For any $U$,$V, W \in TM$ from (\ref{e54}), we have
$$g\big((\overline \nabla_W T)V,U \big)= g(A_{NV} W,U )+ g( t\,h (W,V),U)\,.$$
\noindent Hence, using (\ref{e7}), (\ref{e0}) and (\ref{e6}), we
obtain
$$g((\overline \nabla_W T)V,U)= g(A_{NV} W,U )+ g( h (W,V),NU)\,.$$
\noindent Since $A$ is self-adjoint, with the help of (\ref{e4}), we
get
\begin{equation} \label{e57}
\begin{array}{c}
g((\overline \nabla_W T)V,U)= g(A_{NV} U,W )+ g( A_{NU} V,W).
\end{array}
\end{equation}
Thus, our assertion comes from (\ref{e57}).
\end{proof}
\begin{theorem}
Let $M$ be a proper hemi-slant submanifold of a l.p.R. manifold
$\bar{M}$. If $T$ is parallel, then $M$ is a hemi-slant product. The
converse is true, if $h(\mathcal{D}^{\theta},\mathcal{D}^{\theta})
{\perp } N\mathcal{D}^{\theta} $.
\end{theorem}
\begin{proof}
Let $X$ be in $\mathcal{D}^{\perp} $ and $Z$ in
$\mathcal{D}^{\theta} $. If $T$ is parallel, then from  (\ref{e56}),
we have
\begin{eqnarray} \label{yeni}
A_{NX}Z= -A_{NZ} X.
\end{eqnarray}
Thus, by Corollary 5.4, we conclude that $M$ is a hemi-slant
product. Conversely, if $M$ is a hemi-slant product and
$h(\mathcal{D}^{\theta},\mathcal{D}^{\theta}) {\perp }
N\mathcal{D}^{\theta} $, then for any $Z, W$ and $V \in
\mathcal{D}^{\theta}$, we have $g(A_{NZ} W, V)=g(h(V,W), NZ)=0$. It
means that $A_{ NZ} W \in \mathcal{D}^{\perp}$. Now, let calculate
$g(A_{NZ} W, X)$ for $X\in \mathcal{D}^{\perp} $. Since $ M$ is a
hemi-slant product and $A$ is self-adjoint $g(A_{NZ} W, X)=g(A_{NZ} X, W)=-g(A_{NX} Z, W)=-g(A_{NX} W, Z)
             =-g(A_{NW} X, Z)=-g(A_{NW} Z, X).$

Hence, we deduce
\begin{eqnarray} \label{e58}
A_{NZ} W =-A_{NW} Z\,,
\end{eqnarray}
for all  $Z$, $W \in \mathcal{D}^{\theta}$.

Thus, from (4.13), (\ref{yeni}) and (\ref{e58}), we obtain
(\ref{e56}) and by Theorem 6.1, $T$ is parallel.
\end{proof}
\begin{theorem} Let $M$ be a proper hemi-slant submanifold of $\overline{M}$. If $N$ is parallel, then
\begin{eqnarray}
(a)\quad A_{\mu} {\mathcal{D}^{\perp}}=0\,, \quad(b)\quad
A_{N {\mathcal{D}^{\theta}}} \mathcal{D}^{\perp} =0 , \quad\quad (c)\quad A_{N{\mathcal{D}^{\perp}}} \mathcal{D}^{\theta} =0\,,\qquad\qquad\nonumber\quad\quad\\
(d)\quad \text{M is a hemi-slant product},\quad (e)\quad \text{M
is}\,\, (\mathcal{D}^{\perp}, \mathcal{D}^{\theta})\text{-mixed
totally geodesic}\nonumber.
\end{eqnarray}
\end{theorem}
\begin{proof} Let $N$ be parallel, it follows from (\ref{e55}) that
\begin{equation} \label{e59}
\begin{array}{c}
h(U, TV)= \omega h (U, V)
\end{array}
\end{equation}
for any $ U, V\in TM$. Then, for any $X\in  \mathcal{D}^{\perp}$, we
have
\begin{equation} \label{e60}
\begin{array}{c}
\omega h(U, X)=0
\end{array}
\end{equation}
from (\ref{e59}). For any $\xi \in \mu $, using  (\ref{e6}),
(\ref{e0}) and (\ref{e4}), we have
$$g(\omega h(U, X),\xi) = g ( h(U,X), F \xi)= g(A_{F {\xi}} X, U)\,. $$
Thus, using (\ref{e60}) we get
\begin{eqnarray} \label{e61}
g(A_{F {\xi}} X, U)=0\,.
\end{eqnarray}
Since $\mu$ is invariant with respect to $F$, the assertion (a)
comes from  (\ref{e61}). Now, take $Z \in \mathcal{D}^{\theta}$,
after some calculations, we find
\begin{eqnarray} \nonumber
 g(A_{NZ} X, U)=g(\omega h(U,X),NZ)\,.
\end{eqnarray}

So, using (\ref{e60}), we get $g(A_{NZ} X,U)=0$, which is equivalent
to the assertion (b). On the other hand, for any  $X\in
\mathcal{D}^{\perp}$, using  (\ref{e0}), (\ref{e6}), (\ref{e7}) and
(\ref{e59}), we have
\begin{eqnarray}
&&0=g(h(U,Z),X)=g(Fh(U,Z),FX)=g(\omega h(U,Z),FX)\nonumber\\
&&\quad=g(h(U,TZ),FX)=g(h(U,TZ),NX),\nonumber
\end{eqnarray}

that is, $g(h(U,TZ),NX)=0$. Putting $Z=TZ$ in last
equation, we obtain
$$\cos^2\!\theta\, g(h(U,Z),NX)=\cos^2\!\theta\, g (A_{NX} Z,U)=0\,. $$

\noindent Since $\theta\neq {\pi\over 2}$, the assertion (c)
follows. The assertion (d) follows from the assertions (b), (c) and
(\ref{e51}). Lastly, using (\ref{e141}), from (\ref{e59}),  we have

$\omega^2 h(X,Z)= \omega h(X,TZ)=h(X,T^2Z)=\cos^2\!\theta h(X,Z).$
On the other hand, using (\ref{e8})-(a), we have
$\omega^2h(X,Z)=\omega^2 h(Z,X)=\omega h(Z,TX)=0$. Thus, we get
$\cos^2\!\theta\,h(X,Z)=0$. Since $\theta\neq {\pi\over 2}$, we
deduce that $h(X,Z)=0$, which proves that the last assertion.
\end{proof}
\section{Totally umbilical hemi-slant submanifolds}

In this section we shall give two characterization theorems for the totally umbilical
proper hemi-slant submanifolds of a l.p.R. manifold.
First we prove
\begin{theorem}
If $M$ is a totally umbilical proper hemi-slant submanifold of a
l.p.R. manifold $\bar{M}$, then either the anti-invariant
distribution $\mathcal{D}^{\perp}$ is 1-dimensional or the mean
curvature vector field $H$ of $M$ is perpendicular to $F(
\mathcal{D}^{\perp})$. Moreover, if $M$ is a
 hemi-slant product, then $H\in \mu$.
\end{theorem}
\begin{proof} Since $M$ is a totally umbilical proper hemi-slant submanifold either
$\it {D}im ( \mathcal{D}^{\perp})=1$ or $\it {D}im (
\mathcal{D}^{\perp}) > 1$. If $\it {D}im ( \mathcal{D}^{\perp})=1$,
it is obvious. If $\it {D}im ( \mathcal{D}^{\perp}) > 1$, then we
can choose $X, Y \in \mathcal{D}^{\perp} $ such that $\{ X, Y\}$ is
orthonormal. By using (\ref{e6}), (\ref{e5}), (\ref{e4}) and (4.22),
we have
\begin{eqnarray}
&&g(H, FY)=g(h( X,X),NY)=g(A_{NY} X, X)=0
\end{eqnarray}
It means that
\begin{eqnarray} \label{e62}
H {\perp} F(\mathcal{D}^{\perp}).
\end{eqnarray}
Moreover, if $M$ is a hemi-slant product, for any $Z\in
\mathcal{D}^{\theta}$, using (\ref{e51}) and (2.7), we have
\begin{eqnarray}
&&g(H, NZ)=g(h( X,X),NZ)=g(A_{NZ} X, X)=-g(A_{NX} Z, X)\nonumber\\[.2cm]
             &&\qquad\qquad=-g(h( Z,X),NX)=0.\nonumber
\end{eqnarray}

Hence, it follows that
\begin{eqnarray} \label{e63}
H {\perp} N(\mathcal{D}^{\theta}).
\end{eqnarray}

Thus, using (\ref{e62}) and (\ref{e63}) from (\ref{e12}),  we get
$H\in \mu$.
\end{proof}

Before giving the second result of this section, recall that the following fact about locally product Riemannian manifolds.\\

Let $M_1(c_1)$ (resp. $M_2 (c_2)$) be a real space form with
sectional curvature $c_1$ (resp. $c_2)$. Then the Riemannian
curvature tensor $\overline R$ of the locally product Riemannian
manifold $\bar{M}= M_1(c_1)\times M_2(c_2)$ has the form
\begin{eqnarray} \label{e64}
\qquad\overline R(\bar U,\!\!\bar V)\bar
W\!\!\!\!\!&\!\!\!={1\over 4}
\!(c_1\!\!+\!\!c_2)\!\bigg\{\!g(\bar V\!,\!\bar W)\bar U\!\!-\!\!g(\bar U\!,\!\bar W)\bar V\!\!+\!\!g(F \bar V\!,\!\bar W)F\bar U\!\!-\!\!g(F\bar U,\!\bar W)F\bar V\!\bigg\}\\
             &~+{1\over 4}\!(c_1\!\!-\!\!c_2)\bigg\{\!g(F\bar V,\!\bar W)\bar U\!\!-\!\!g(F\bar U\!,\!\bar W)\bar V\!+\!\!g(\bar V\!,\!\bar W)F\bar U\!\!-\!\!g(\bar U\!,\!\bar W)F\bar V \bigg \},\nonumber
\end{eqnarray}
where $\bar U, \bar V,\bar W \in T\bar M$ \,\, \cite{Yan}.

\begin{theorem} Let $M$ be a totally umbilical hemi-slant submanifold with parallel mean curvature vector field
$H$ of a l.p.R. manifold $\bar{M } =M_1(c_1)\times M_2(c_2)$ with $c_1\neq c_2$. Then, $M$ can not be proper.
\end{theorem}
\begin{proof} Let $X \in \mathcal{D}^{\perp} $ and $ Z\in \mathcal{D}^{\theta} $ be two unit vector fields.
Since $H$ is parallel, using (\ref{a1}) and (\ref{e5}) from the
Codazzi equation (2.9), we have
\begin{eqnarray} \label{e66}
({\overline R }(X, Z) X)^{\perp} = -\nabla _Z ^{\perp}H=0.
\end{eqnarray}
On the other hand, the equation (\ref{e64}) gives
\begin{eqnarray} \label{e67}
{\overline R }(X, Z) X = -{1\over 4} \bigg \{ (c_1+ c_2) Z+(c_1-c_2) FZ \bigg \}.
\end{eqnarray}
Taking the normal component of (\ref{e67}), we get
\begin{eqnarray} \label{e68}
({\overline R }(X, Z) X)^{\perp} = -{1\over 4} (c_1- c_2)NZ ,
\end{eqnarray}
which contradicts (\ref{e66}).
\end{proof}

\noindent
We have immediately from Theorem 7.2. that:
\begin{cor} There exists no totally geodesic proper hemi-slant submanifold of a l.p.R. manifold
$\bar{M } =M_1(c_1)\times M_2(c_2)$ with $c_1\neq c_2$.
\end{cor}
\section{Ricci curvature of  hemi-slant submanifolds}
In this section, we obtain a basic inequality involving Ricci
curvature and the squared mean curvature of a hemi-slant submanifold
of a l.p.R. manifold $\bar{M}=M_1(c_1)\times M_2(c_2)$. 
We first represent the following fundamental facts about this topic.\\

Let $\bar{M}$ be a $n$-dimensional Riemannian manifold equipped with a Riemannian metric $g$ and
$\{e_{1},...,e_{n}\}$ be an orthonormal basis for $T_{p}\bar{M},$ $p\in \bar{M}.$ Then the \emph{Ricci tensor} $\overline S$
is defined by
\begin{eqnarray} \label{e68}
\overline S(U,V)=\displaystyle\sum^{n}_{i=1}\overline R(e_{i},U,V,e_{i}),
\end{eqnarray}
where $U,V\in T_{p}\bar{M}.$ For a fixed $i\in\{1,...,n\}$, the \emph{Ricci curvature} of $e_{i}$, denoted by $\overline Ric(e_{i}),$ is given by
\begin{eqnarray} \label{e69}
\overline Ric(e_{i})=\displaystyle\sum^{n}_{i\neq j}\overline K_{ij},
\end{eqnarray}
where $\overline K_{ij}=g(\overline R(e_{i},e_{j})e_{j},e_{i})$ is the \emph{sectional curvature} of the plane spanned by the plane spanned by $e_{i}$ and $e_{j}$ at $p\in \bar{M}.$
Let $\Pi_{k}$ be a $k$-plane of $T_{p}\bar{M}$ and $\{e_{1},...,e_{k}\}$ any orthonormal basis of $\Pi_{k}$.
For a fixed $i\in\{1,...,k\}$, the $k$-\emph{Ricci curvature} \cite{Chen} of $\Pi_{k}$ at $e_{i}$, denoted by $\overline Ric_{\Pi_{k}}(e_{i}),$ is defined  by

\begin{eqnarray} \label{e70}
\overline Ric_{\Pi_{k}}(e_{i})=\displaystyle\sum^{k}_{i\neq j}\overline K_{ij}.
\end{eqnarray}
It is easy to see that $\overline Ric_{({T_{p}\!\bar{M}})}(e_{i})=\overline Ric(e_{i})$ for $1\leq i\leq n,$ since $\Pi_{n}=T_{p}\bar{M}.$\\

We now recall that the following basic inequality [10, Theorem 3.1] involving Ricci curvature and the squared mean curvature of a submanifold
of a Riemannian manifold.
\begin{theorem} ([10, Theorem 3.1]) Let $M$ be an $m$-dimensional  submanifold
of a Riemannian manifold $\bar{M}$. Then, for any unit vector $X\in T_{p}M$, we have
\begin{eqnarray} \label{e71}
Ric(X)\leq\frac{1}{4}m^{2}\|H\|^{2}+\overline Ric_{({T_{p}\!M})}(X)
\end{eqnarray}
where $Ric(X)$ is the Ricci curvature of $X$.
\end{theorem}
Of course, the equality case of (\ref{e71}) was also discussed in \cite{Ho}, but we will not deal with the equality case in this paper.\\

Now, we are ready to state main result of this section.
\begin{theorem} Let $M$ be an $m$-dimensional  hemi-slant submanifold of a l.p.R. manifold $\bar{M}=M_1(c_1)\times M_2(c_2).$
 Then, for unit vector $V\in T_{p}M$, we have
\begin{eqnarray} \label{e72}
4Ric(V)\leq m^{2}\|H\|^{2}\!+(c_{1}\!+c_{2})\bigg\{\!(m-1)\!+\!\displaystyle\sum^{m}_{i=2}\!g(Te_{i},e_{i})g(TV,V)\\
-\|TV\|^{2}\!\!+\!g(TV,V)\!\bigg\}\!+\!(c_{1}\!-\!c_{2}\!)\bigg\{\!\!\displaystyle\sum^{m}_{i=2}\!g(Te_{i},e_{i})\!+\!\!(m\!-\!1\!)g(TV,V)\!\bigg\}\nonumber
\end{eqnarray}
where $\{V,e_{2},...,e_{m}\}$ is an orthonormal basis for $T_{p}M.$
\end{theorem}

\begin{proof} Let $M$ be an $m$-dimensional  hemi-slant submanifold of a l.p.R. manifold $\bar{M}=M_1(c_1)\times M_2(c_2).$ Then
for any unit vector $V\in T_{p}M$, using (\ref{e64}) and (\ref{e6})
from (\ref{e70}) we have
\begin{eqnarray} \label{e73}
\!\!\!\!\!\!\!\!4\overline Ric_{(\!T_{p}\!M\!)}\!(V)\!\!=\!\!(c_{1}\!+c_{2})\bigg\{\!(m-1)\!+\!\displaystyle\sum^{m}_{i=2}\!g(Te_{i},e_{i})g(TV,V)\\
-\|TV\|^{2}\!\!+\!g(TV,V)\!\bigg\}\!+\!(c_{1}\!-\!c_{2}\!)\bigg\{\!\!\displaystyle\sum^{m}_{i=2}\!g(Te_{i},e_{i})\!+\!\!(m\!-\!1\!)g(TV,V)\!\bigg\}\nonumber
\end{eqnarray}
Thus, using (\ref{e73}) in (\ref{e71}) we get (\ref{e72}).
\end{proof}
\begin{remark} In general, $g(F\overline{V},\overline{V})\neq0$ for any unit vector $\overline{V}\in T_{p}\bar{M}$
in a l.p.R. manifold $\bar{M}$, contrary to almost Hermitian
$(g(J\overline{V},\overline{V})=0)$ and almost contact
$((g(\varphi\overline{V},\overline{V})=0)$ manifolds. However, we
can establish that the almost product structure $F$ in a l.p.R.
manifold $\bar{M}$ such that $g(F\overline{V},\overline{V})=0,$ for
all  $\overline{V}\in T_{p}\bar{M}$. In fact, if $\bar{M}$ is an
even dimensional l.p.R. manifold with an orthonormal basis
$\{e_{1},...,e_{n},e_{n+1},...,e_{2n}\}$, then we can define $F$ by
\begin{equation}\nonumber
F(e_{j})=e_{n+j}, \quad F(e_{n+j})=e_{j},\quad j\in\{1,2,...,n\}.
\end{equation}
Hence, we observe easily that the almost product structure $F$
satisfies
\begin{eqnarray} \label{e74}
g(Fe_{j},e_{j})=0.
\end{eqnarray}
\end{remark}
For example, the almost product structure $F$ in example of section 3, satisfies the condition (\ref{e74}).
On the other hand, because of Lemma 3.3 and the equation (3.5), we have $TV=0$, if $V\in\mathcal{D}^{\perp}$
and $\|TV\|^{2}=\cos^{2}\!\theta$, if $V\in\mathcal{D}^{\theta}$ and $\|V\|=1$, respectively. Thus, by Theorem 8.2 we get the following two results.
\begin{cor} Let $M$ be an $m$-dimensional  anti-invariant submanifold of a l.p.R. manifold $\bar{M}=M_1(c_1)\times M_2(c_2).$
If the almost product structure $F$ of $\bar{M}$ satisfies the
condition (\ref{e74}), then we have
\begin{eqnarray} \nonumber
4Ric(V)\leq m^{2}\|H\|^{2}+(c_{1}+c_{2})(m-1),
\end{eqnarray}
where $V\in T_{p}M$ is any unit vector.
\end{cor}

\begin{cor} Let $M$ be an $m$-dimensional slant submanifold of a l.p.R. manifold $\bar{M}=M_1(c_1)\times M_2(c_2).$
If the almost product structure $F$ of $\bar{M}$ satisfies the
condition (\ref{e74}), then we have
\begin{eqnarray} \nonumber
4Ric(Z)\leq m^{2}\|H\|^{2}+(c_{1}+c_{2})\{(m-1)-\cos^{2}\!\theta\},
\end{eqnarray}
where $Z\in T_{p}M$ is any unit vector.
\end{cor}

\bibliographystyle{amsplain}

\end{document}